\documentclass[preprint,12pt,3p]{elsarticle}
\usepackage{amssymb} % The amssymb package provides various useful mathematical symbols
\usepackage{amsmath} % The amsmath package provides various useful equation environments.
\usepackage{amsthm} % The amsthm package provides extended theorem environments
\usepackage{xfrac}
\usepackage{float}
\usepackage{xcolor} 
\usepackage{soul}
\usepackage{natbib}
\usepackage{tcolorbox}
\usepackage{hyperref}
\usepackage[ruled,linesnumbered]{algorithm2e}
\numberwithin{equation}{section}
\newtheorem{thm}{Theorem}[section]
\newtheorem{defi}{Definition}[section]
\newtheorem{lem}{Lemma}[section]
\newtheorem{pro}{Proposition}[section]
\newtheorem{remark}{Remark}

\newtheorem{assum}{Assumption}
\DeclareMathOperator{\Tr}{Tr}
\DeclareMathOperator{\Vect}{vec}
\DeclareMathOperator{\e}{exp}
\DeclareMathOperator{\Vech}{vech}
\DeclareMathOperator{\Max}{max}
\DeclareMathOperator{\Min}{min}
\DeclareMathOperator{\diag}{diag}

\journal{****}

\begin{document}

	\begin{frontmatter}
		
		\title{Computing stabilizing feedback gains for stochastic linear systems via policy iteration method}
		\author[label1]{Xinpei Zhang} 
		\ead{zhangxinpei@mail.sdu.edu.cn}
		
		\author[label1,label2]{Guangyan Jia\corref{cor1}} 
		\ead{jiagy@sdu.edu.cn}
		
		\cortext[cor1]{Corresponding author}
		
		\affiliation[label1]{organization={Zhongtai Securities Institute for Financial Studies, Shandong University},
			addressline={27 Shanda Nanlu},
			city={Jinan},
			postcode={250100}, 
			state={Shandong},
			country={P.R.China}}
		
		\affiliation[label2]{organization={School of Mathematics, Shandong University},
			addressline={27 Shanda Nanlu},
			city={Jinan},
			postcode={250100}, 
			state={Shandong},
			country={P.R.China}}

		%% Abstract
		\begin{abstract}
			In recent years, stabilizing unknown dynamical systems has became a critical problem in control systems engineering. Addressing this for linear time-invariant (LTI) systems is an essential fist step towards solving similar problems for more complex systems. In this paper, we develop a model-free reinforcement learning algorithm to compute stabilizing feedback gains for stochastic LTI systems with unknown system matrices. This algorithm proceeds by solving a series of discounted stochastic linear quadratic (SLQ) optimal control problems via policy iteration (PI). And the corresponding discount factor gradually decreases according to an explicit rule, which is derived from the equivalent condition in verifying the stabilizability. We prove that this method can return a stabilizer after finitely many steps. Finally, a numerical example is provided to illustrate the effectiveness of the proposed method.
		\end{abstract}

		%% Keywords
	\end{frontmatter}

	%% main text
	\section{Introduction}\label{sec1}
		Over the past few years, reinforcement learning has made significant advances in solving stochastic optimal control problems (\citet{sutton1998reinforcement,bertsekas2019reinforcement}), especially in solving infinite-horizon SLQ optimal control problems where drift and diffusion terms in the dynamics involve the state and control. Related work includes: \citet{Zhang2025Convergence} solved such problems with random initial state via gradient method under full system knowledge; \citet{li2022stochastic} proposed an online PI algorithm to obtain the optimal controller for infinite-horizon SLQ problems with partial system information; Based on the adaptive dynamic programming, \citet{zhang2023adaptive} extended the result of \citet{li2022stochastic} to the case where all system coefficient matrices are unknown; among others. Note that nearly all these papers assume that an initial stabilizing feedback gain is known. However, obtaining stabilizers is known to be challenging in the model-free setting (\citet{zhang2023adaptive,jiang2012computational}). Consequently, the dependence of initial stabilizers significantly limits the application of these reinforcement learning algorithms. From this, at the present stage, synthesizing an initial stabilizer for SLQ problems emerges as a critical problem in control systems engineering.
		
		In this background, this paper is devoted to the computation of stabilizers for stochastic LTI systems with unknown dynamics matrices. The idea arises from the fact that the stabilizing feedback gains are much easier to obtain for the discounted SLQ problems with large discount factors. Further, as the discount factor $\alpha$ decreases, the domain of the corresponding infinite-horizon discounted SLQ problem progressively converges to the set of all stabilizers for the original stochastic LTI system. Guided by these observations, our algorithm starts from a stabilizer for highly discounted problems, then alternates iteratively between updating the policy via PI and decreasing the discount factor while ensuring the stability. This algorithm terminates when $\alpha \leq 0$, yielding a stabilizing feedback gain for the LTI system.
		
		Our work is inspired by the recently developed discount method, which is a class of system synthesis methods built upon discounted optimal control problems with varying discount factors. This method was originally developed for escaping locally optimal policy in multi-agent control systems (\citet{feng2020escaping,feng2021damping}). Subsequently, it was applied to compute a stabilizing feedback gain for discrete-time and continuous-time linear quadratic regulator (LQR) problems with the random initial state. \citet{perdomo2021stabilizing} stabilized both linear and smooth nonlinear discrete-time systems by alternating between obtaining a near-optimal policy via policy gradient and finding a discount factor via binary or random search. A more closely related work to this paper comes from \citet{lamperski2020computing}. It synthesized a stabilizing linear feedback control for discrete-time LQR problems based on PI. However, both aforementioned works require a search procedure for the discount factor. This limitation was addressed by \citet{zhao2024convergence}. They designed an explicit rule to adjust the discount factor. Further, they established the sample complexity of policy gradient methods for data-driven stabilization of discrete-time LTI systems. In addition, \citet{ozaslan2022computing} used the discount method to stabilize continuous-time LQR problems. By updating policy via policy gradient methods, they kept the cost value below a uniform threshold, and thus the finite-time convergence guarantee was provided.
		
		In this paper, different from the work of \citep{perdomo2021stabilizing}, \citep{ lamperski2020computing}, \citep{zhao2024convergence} and \citep{ozaslan2022computing}, we study more complicated It$\hat{\text{o}}$-type stochastic LTI systems where drift and diffusion terms are affected by both the state and control. We propose an off-policy model-free algorithm in which an explicit update rule for decreasing the discount factor is designed by using the equivalent condition in verifying the stabilizability of stochastic LTI systems. This rule provides a uniform lower bound for decrement of the discount factor in each iteration, thereby keeping the total iterations of algorithm finite. Moreover, what is worth mentioning is that the proposed algorithm can synthesize stabilizers for stochastic LTI systems with both deterministic and random initial states.
		
		The rest of this article is organized as follows. In section \ref{sec2}, we describe the stabilizability of stochastic LTI systems with randomized initial states and introduce the discounted SLQ problems with the discount factor $\alpha \geq 0$. In section \ref{sec3}, we propose the discount method to compute stabilizing feedback gains for SLQ problems with both the known and completely unknown system matrices and discuss the feasibility of this method. In addition, we extend this method to stabilize SLQ with the deterministic initial state. Finally, a numerical example is shown in section \ref{sec4}.
		
		\textbf{Notation} We denote by $\mathbb{R}^n$ the $n$-dimensional Euclidean space with the norm $|\cdot|$. Let $\mathbb{R}^{n \times m}$ denote the space of all $(n \times m )$ real matrices. Let $\mathbb{S}^{n}$ denote the set of all $(n \times n)$ real symmetric matrices. The set of all $(n \times n)$ positive definite (resp., positive semi-definite) matrices is denoted by $\mathbb{S}_{+}^{n}$ (resp., $\overline{\mathbb{S}_{+}^{n}}$). We use $\Tr(\cdot)$ to denote the trace of a square matrix. We use $\| \cdot \|_2$ and $\| \cdot \|$ to denote the spectral norm and the Frobenius norm of a matrix, respectively. Let $\lambda_i(\cdot)$ denote the $i$-th smallest eigenvalue of a matrix. Let $A \otimes B$ denote the Kronecker product of matrices $A$ and $B$. We denote by $\Vect{(\cdot)}$ the vectorization of a matrix, which obtained by stacking the columns of the matrix on top of one another. In addition, if $A \in \mathbb{S}_{+}^{n}$ (resp., $A \in \overline{\mathbb{S}_{+}^{n}}$) is a positive definite (resp., positive semi-definite) matrix, we write $A\succ 0$ (resp., $A\succeq 0)$. For any $A, B \in \mathbb{S}^{n}$, we use the notation $A\succ B$ (resp., $A\succeq B)$ to indicate that $A-B\succ 0$ (resp., $A-B\succeq 0)$. 
		
	\section{Problem formulation and preliminaries}\label{sec2}
		\noindent
		\textbf{A. Problem formulation}	
		
		Let $\left(\Omega,\mathcal{F},\mathbb{F},\mathbb{P}\right)$ be a complete filtered probability space on which a standard one-dimensional Brownian motion $W = \left\{W(t)|t \geq 0\right\}$ is defined, where $\mathbb{F} = \left\{ \mathcal{F}_{t} \right\}_{t \geq 0}$ is the natural filtration of $W$ augmented by all the $\mathbb{P}\text{-null}$ sets in $\mathcal{F}$ and an independent $\sigma$-algebra $\mathcal{H}$. In this paper, we consider the following time-invariant stochastic linear system:
		\begin{equation}\label{2.1}
			\begin{cases}
				dX(t) = [AX(t)+Bu(t)]dt + [CX(t)+Du(t)]dW(t),  \quad t\geq 0, \\
				X(0) = \xi_0 \in \mathcal{H},
			\end{cases}
		\end{equation}
		where $X(\cdot)$ is called the state process valued in $\mathbb{R}^n$ with the initial state $\xi_0$ being a $\mathcal{H}$-measurable random variable; $u(\cdot)$ is called the control process valued in $\mathbb{R}^m$. The coefficients $A, C \in \mathbb{R}^{n \times n}$, and $B, D \in \mathbb{R}^{n \times m}$ are constant matrices. Here, the dimension of Brownian motion $W$ is set to be 1 for simplicity. In addition, we briefly denote the above state system (\ref{2.1}) by $[A,C;B,D]$.
		
		\begin{assum}
			For the initial state $X(0)$, we assume $\Sigma_0 := \mathbb{E} X(0)X^{\top}(0)$ is positive-definite.
		\end{assum}
		
		\begin{remark}
			The motivation for using a random initial state $X(0)$ and assuming $\Sigma_0 := \mathbb{E} X(0) X(0)^\top\succ 0$ is to ensure both the well-definedness of $1/\lambda_{1}(\Sigma_0)$ in Lemma \ref{pro3.1} and a strictly positive decrement $\Delta \alpha$ in Algorithm \ref{alg:Model-based discount} and \ref{alg:Model-free discount}.
		\end{remark}
	
		\noindent
		\textbf{B. Mean-square stabilizable}	
		
		\begin{defi}\cite{rami2000linear,sun2020stochastic}\label{defi2.1}
			The system $[A,C; B,D]$ is called mean-square stabilizable if there exists a constant matrix $K \in \mathbb{R}^{m \times n}$, for every initial state $X(0)$, the solution of the following equation
			$$
			dX(t)=(A+BK)X(t)dt+(C+DK)X(t)dW(t) 
			$$
			satisfies $\lim\limits_{t\to+\infty}\mathbb{E}[X(t)^{\top}X(t)]=0$.
			
			In this case, $K$ is called a (mean-square) stabilizer of the system $[A,C; B,D]$, and the feedback control $u(\cdot)=KX(\cdot)$ is called (mean-square) stabilizing. The set of all mean-square stabilizers of $[A,C; B,D]$ is denoted by $\mathcal{K} \equiv \mathcal{K}\left([A,C; B,D]\right)$.
		\end{defi}
		
		Without loss of generality, we assume that the system $[A,C; B,D]$ is mean-square stabilizable, i.e. $\mathcal{K} \neq \emptyset$. The following lemma provides an equivalent characterization for the mean-square stabilizers. For a proof, see (\citet{rami2000linear}, Theorem 1).
		
		\begin{lem}\label{lem2.1}
			A matrix $K \in \mathbb{R}^{m \times n}$ is a stabilizer of the system $[A,C; B,D]$ if and only if there exists a $P \in \mathbb{S}^{n}_{+}$ such that
			$$
			(A+BK)^\top P+ P (A+BK) + (C+DK)^\top P (C+DK) \prec 0.
			$$
			In this case, for any $\Lambda \in \mathbb{S}^{n}$ (respectively, $\Lambda \in \overline{\mathbb{S}_{+}^{n}}$, $\Lambda \in \mathbb{S}^{n}_{+}$), there exists a unique solution $P \in \mathbb{S}^{n}$  (respectively, $P \in \overline{\mathbb{S}_{+}^{n}}$, $P \in \mathbb{S}^{n}_{+}$) to the following matrix equation:
			$$
			(A+BK)^\top P+ P (A+BK) + (C+DK)^\top P (C+DK) + \Lambda = 0.
			$$
		\end{lem}
		
		\noindent
		\textbf{C. The discounted SLQ optimal control problem}
		
		For the discount factor $\alpha \geq 0$, the discounted SLQ optimal control problem is defined as
		\begin{equation}\label{2.2}
			\begin{aligned} 
				&\Min  \: \mathbb{E} \Big[ \int_{0}^{+\infty} \e(-2\alpha t) \cdot  \left(X(t)^{\top}QX(t) + u(t)^{\top}Ru(t) \right)dt \Big],  \\
				&\text{subject to } \left(\ref{2.1}\right),
			\end{aligned}	
		\end{equation}
		where $Q \in \mathbb{S}_{+}^{n}$ , $ R\in \mathbb{S}_{+}^{m}$ are given constant matrices.
		
		Let $\widetilde{X}(t) = \e(-\alpha t)X(t)$ and $\widetilde{u}(t) = \e(-\alpha t)u(t)$, by It$\hat{\text{o}}$ formula, 
		\begin{equation}\label{2.3}
			d\widetilde{X}(t) = [A_{\alpha}\widetilde{X}(t)+ B \widetilde{u}(t)]dt + [C\widetilde{X}(t) + D\widetilde{u}(t)]dW(t), 
		\end{equation}
		where $A_{\alpha} := A - \alpha I_n$. We denote this system (\ref{2.3}) by $[A_\alpha, C; B,D]$ for every discount factor $\alpha \geq 0$. 
		
		By introducing exponentially weighted state $\widetilde{X}(t)$ and input $\widetilde{u}(t)$, the discounted SLQ problem (\ref{2.2}) is equivalent to the following undiscounted SLQ problem:
		\begin{equation*}
			\begin{aligned} 
				&\Min \: \mathbb{E} \Big[ \int_{0}^{+\infty} \widetilde{X}(t)^{\top}Q\widetilde{X}(t) + \widetilde{u}(t)^{\top} R \widetilde{u}(t) dt  \Big],  \\
				&\text{subject to } \left(\ref{2.3} \right) \text{ and } \widetilde{X}(0)= \xi_0 \in \mathcal{H} . 
			\end{aligned}
		\end{equation*}
		Consequently, some properties of the standard SLQ problem extend directly to its discounted counterpart. In addition, in Section \ref{sec3},	our analysis is building upon the above equivalence relation.
				
		Given that this paper aims to compute stabilizers for stochastic LTI systems within a reinforcement learning (RL) framework, we, unless otherwise specified, confine our analysis to the following linear state-feedback control
		$$
			\widetilde{u}(\cdot) = K \widetilde{X}(\cdot),
		$$
		where the policy is linearly parameterized by the constant matrix $K \in \mathbb{R}^{m \times n}$. Now the state dynamics can be written as
		\begin{equation}\label{2.4a}
			d\widetilde{X}(t) = \left(A_{\alpha}+ BK\right)\widetilde{X}(t)dt + \left(C+ DK\right)\widetilde{X}(t)dW(t)
		\end{equation}
		and the corresponding cost function is denoted as
		\begin{equation}\label{2.4}
			J_\alpha(K):= \mathbb{E} \int_{0}^{+\infty} \widetilde{X}(t)^{\top}\left( Q + K^\top RK \right)\widetilde{X}(t)  dt. 
		\end{equation}
		
		We denote by $\mathcal{K}^{(\alpha)}$ the set of all stabilizers of the system $[A_\alpha, C; B,D]$. The feasible set $\mathcal{K}^{(\alpha)}$ shrinks as parameter $\alpha$ decreases. The following lemma shows this result in detail.
		\begin{lem}\label{lem2.2}
			For $\alpha_1, \alpha_2 \geq 0$, if $\alpha_1 \leq \alpha_2$, then $\mathcal{K}^{(\alpha_1)} \subseteq \mathcal{K}^{(\alpha_2)}$.
		\end{lem}
		\begin{proof}
			For any $K_{\alpha_1} \in \mathcal{K}^{(\alpha_1)}$ and $\Lambda_{\alpha_1} \in\mathbb{S}_{+}^{n}$, it follows from Lemma \ref{lem2.1} that there exists a unique solution $P_{\alpha_1} \in \mathbb{S}_{+}^{n}$ to the following matrix equation:
			$$
			(A_{\alpha_1}+BK_{\alpha_1})^\top P_{\alpha_1}+ P_{\alpha_1} (A_{\alpha_1}+BK_{\alpha_1}) + (C+DK_{\alpha_1})^\top P_{\alpha_1} (C+DK_{\alpha_1}) + \Lambda_{\alpha_1} = 0.
			$$	
			Denote $\Delta \alpha_1 := \alpha_2-\alpha_1 \geq 0$, then 
			\begin{equation*}
				\begin{aligned}
					&(A_{\alpha_2}+BK_{\alpha_1})^\top P_{\alpha_1}+ P_{\alpha_1} (A_{\alpha_2}+BK_{\alpha_1}) + (C+DK_{\alpha_1})^\top P_{\alpha_1} (C+DK_{\alpha_1}) \\
					&\quad = (A_{\alpha_1}+BK_{\alpha_1})^\top P_{\alpha_1}+ P_{\alpha_1} (A_{\alpha_1}+BK_{\alpha_1}) \\
					&\quad \quad + (C+DK_{\alpha_1})^\top P_{\alpha_1} (C+DK_{\alpha_1}) -2\Delta \alpha_1 P_{\alpha_1} \\
					&\quad = -\Lambda_{\alpha_1} -2\Delta \alpha_1 P_{\alpha_1} \prec 0
				\end{aligned}
			\end{equation*}	
			Here, the last partial order follows from the positive definiteness of matrices $\Lambda_{\alpha_1}$ and $P_{\alpha_1}$. Thus, by Lemma \ref{lem2.1}, $K_{\alpha_1}$ is a stabilizer of the system $[A_{\alpha_2},C; B,D]$, i.e., $K_{\alpha_1} \in \mathcal{K}^{\left(\alpha_2\right)}$. Therefore, $\mathcal{K}^{\left(\alpha_1\right)} \subseteq \mathcal{K}^{\left(\alpha_2\right)}$.			 
		\end{proof}
		
		\begin{lem}\label{lem2.3}
			For $0 \leq \alpha_1 \leq \alpha_2$, it holds that 
			$
			J_{\alpha_1}^{*} \geq J_{\alpha_2}^{*},
			$
			where $J_{\alpha}^{*}$ denotes the optimal cost value $J_{\alpha}^{*}:=  \Min_{K \in \mathcal{K}^{(\alpha)}}J_\alpha(K)  $.
		\end{lem}
		\begin{proof}
			From (\ref{2.4}), for arbitrary $K \in \mathcal{K}^{(\alpha_2)}$, it holds that
			$$
			J_{\alpha_1}(K) \geq J_{\alpha_2}(K) \geq \underset{K \in \mathcal{K}^{(\alpha_2)}}\Min  J_{\alpha_2}(K).
			$$
			Since $K \in \mathcal{K}^{(\alpha_2)}$ is arbitrary and $\mathcal{K}^{(\alpha_1)} \subseteq \mathcal{K}^{(\alpha_2)}$ (Lemma \ref{lem2.2}), we have
			$$
			\underset{K \in \mathcal{K}^{(\alpha_1)}}\Min  J_{\alpha_1}(K) \geq \underset{K \in \mathcal{K}^{(\alpha_2)}}\Min  J_{\alpha_2}(K),
			$$
			which completes the proof.
		\end{proof}
		
		Additionally, it follows from (Lemma 5, \citet{rami2000linear}), the objective function (\ref{2.4}) can also be written as
		\begin{equation}\label{2.5}
			J_\alpha (K) = \Tr \left(P_\alpha \Sigma_0 \right)
		\end{equation}
		when $K \in \mathcal{K}^{\left(\alpha\right)}$, where $P_\alpha$ is the solution of the following Lyapunov equation:
		\begin{equation}\label{2.6}
			\begin{aligned}
				(A_{\alpha}&+BK)^\top P_\alpha + P_\alpha (A_{\alpha}+BK)  \\
				&+ (C+DK)^\top P_\alpha (C+DK) + Q + K^\top R K =0.
			\end{aligned}
		\end{equation}
		
	\section{Stabilizing linear systems via discount method}\label{sec3}
		\noindent
		\textbf{A. Known model}
		
		We first describe how the proposed algorithm provably synthesizes a stabilizer $K \in \mathcal{K}$ for the system (\ref{2.1}) with known system matrices. In a nutshell, this algorithm is achieved by reducing the stabilization problem  to solving a sequence of discounted SLQ problems via the PI algorithm. We start by choosing a sufficiently large initial discount factor $\alpha_0$ such that $\mathbf{O}_{m \times n} \in \mathcal{K}^{(\alpha_0)}$. 
		
		\begin{lem}\label{lem3.1}
			For the dynamical system $[A_{\alpha_0},C; B,D]$, if the parameter $\alpha_0$ satisfies 
			$$
			\alpha_0 > \frac{1}{2}\left(\lambda_{n}(A+A^\top) + \|C\|_2^2\:\right),
			$$
			then $\mathbf{O}_{m \times n}$ is a  mean-square stabilizer of the system $[A_{\alpha_0},C; B,D]$, i.e. $\mathbf{O}_{m \times n}\in \mathcal{K}^{(\alpha_0)}$.
		\end{lem}
		
		\begin{proof}
			When $K = \mathbf{O}_{m \times n}$, one gets
			\begin{equation*}
				\begin{aligned}
					&(A_{\alpha_0}+BK)^\top P_{\alpha_0}+ P_{\alpha_0} (A_{\alpha_0}+BK) + (C+DK)^\top P_{\alpha_0} (C+DK) \\
					&\quad = A_{\alpha_0}^\top P_{\alpha_0} + P_{\alpha_0}A_{\alpha_0} + C^\top P_{\alpha_0}C \\
					&\quad = A^\top P_{\alpha_0} + P_{\alpha_0}A + C^\top P_{\alpha_0}C -2\alpha_0P_{\alpha_0}.
				\end{aligned}
			\end{equation*}
			If $\alpha_0 > \frac{1}{2} \left(\lambda_{n}(A^\top+A)+ \|C\|_2^2\right)$, we have 
			$$
			\lambda_{n}(A^\top+A + C^\top C -2\alpha_0I_n) \leq \lambda_{n}(A^\top+A)+ \|C\|_2^2 -2\alpha_0 <0.
			$$ 
			Here, the first inequality follows from Lemma \ref{lemA.3}. Then $P_{\alpha_0} = I_n \succ 0$ is a solution of the following inequality:
			$$
			A^\top P_{\alpha_0} + P_{\alpha_0}A + C^\top P_{\alpha_0}C -2\alpha_0P_{\alpha_0} \prec 0.
			$$
			Thus, by Lemma \ref{lem2.1}, $\mathbf{O}_{m \times n}\in \mathcal{K}^{(\alpha_0)}$
		\end{proof}
		
		\begin{remark}
			For the discounted SLQ problem with discount factor $\alpha_0$, one can solve it by using the PI method initialized at $\mathbf{O}_{m \times n} \in \mathcal{K}^{(\alpha_0)}$. The theoretical analysis established by \citet{li2022stochastic} guarantees that the generated policy sequence converges to the corresponding optimal policy $K_{\alpha_0}^{*}$.
			
			With regard to the optimal policy $K_{\alpha_0}^{*}$, an interesting question might be whether it stabilizes the original LTI system $[A,C;B,D]$. Unfortunately, the answer is in the negative. A specific example is stated as follows.
			
			Set  
			\begin{equation*}
				A = 
				\begin{bmatrix} 
					4 & 7 \\ 
					5 & -13
				\end{bmatrix},
				B = 
				\begin{bmatrix} 
					6 \\ 
					1
				\end{bmatrix},
				C = 
				\begin{bmatrix} 
					5 & -1 \\ 
					-3 & 4
				\end{bmatrix},
				D = 
				\begin{bmatrix} 
					2 \\
					8
				\end{bmatrix},
				\Sigma_0 = 
				\begin{bmatrix} 
					1 & 0 \\ 
					0 & 1
				\end{bmatrix}.
			\end{equation*}
			The coefficients in cost functional are chosen as
			\begin{equation*}
				Q = 
				\begin{bmatrix} 
					6 & 0 \\ 
					0 & 3
				\end{bmatrix}, \quad
				R = 2.
			\end{equation*}
			
			By Lemma \ref{lem3.1}, we set $\alpha_0 = 29$. Then implementing the PI algorithm (\citet{li2022stochastic}), one gets the optimal policy $K_{\alpha_0}^{*} = (-0.41059,-0.17726)$. 
			
			It follows from (Remark 1, \citet{rami2000linear}) that the mean-square stabilizability of the system $[A,C;B,D]$ implies the stabilizability of the pair $[A,B]$ in the deterministic sense. However, in this case, 
			$$
				\lambda_n(A+BK_{\alpha_0}^{*}) > 0.
			$$
			Therefore, policy $K_{\alpha_0}^{*}$ fails to stabilize the original linear system $[A,C;B,D]$.
		\end{remark}
	
		Based on this result, we subsequently present an explicit rule to identify the decrement $\Delta \alpha$ of discount factor $\alpha$. It follows from Lemma \ref{lem2.2} that the set $\mathcal{K}^{(\alpha)}$ shrinks as parameter $\alpha$ decrease, and this may result in an originally stabilizing feedback gain $K \in \mathcal{K}^{(\alpha)}$ no longer being an element of the contracted $\mathcal{K}^{(\alpha - \Delta \alpha)}$. To avoid this situation, the following lemma provides a selection guideline for $\Delta \alpha$.
		
		\begin{lem}\label{pro3.1}
			For any $K \in \mathcal{K}^{\left(\alpha\right)}$, $\alpha \geq 0$ and $\zeta > 1$, if the non-negative decrement $\Delta \alpha$ satisfies 
			\begin{equation}\label{3.1}
				\Delta \alpha \leq \frac{\lambda_{1}(\Sigma_0)\lambda_{1}(Q )}{2J_{\alpha}(K)} \frac{\zeta-1}{\zeta}, 
			\end{equation}
			then $$K \in \mathcal{K}^{\left( \alpha-\Delta \alpha \right)} \quad \text{and} \quad
			J_{\alpha-\Delta \alpha}(K) \leq \zeta J_\alpha (K).$$
		\end{lem}
		\begin{proof}
			We first show that for any $K \in \mathcal{K}^{\left(\alpha\right)}$ and a scalar $\zeta >1$ such that (\ref{3.1}) holds, $K \in \mathcal{K}^{\left(\alpha ^\prime\right)}$, where $\alpha^\prime := \alpha - \Delta \alpha$. To this end, we first notice that 
			\begin{equation*}
				\begin{aligned}
					& (A_{\alpha^\prime}+BK)^\top P_\alpha + P_\alpha (A_{\alpha^\prime}+BK) + (C+DK)^\top P_\alpha (C+DK) \\
					&\quad = (A_{\alpha}+BK)^\top P_\alpha + P_\alpha (A_{\alpha}+BK) \\
					&\quad \quad + (C+DK)^\top  P_\alpha(C+DK) + 2 \Delta \alpha P_\alpha \\
					&\quad = -(Q + K^\top RK) + 2\Delta \alpha P_\alpha,
				\end{aligned}
			\end{equation*}
			where $P_{\alpha} \in \mathbb{S}^{n}_{+}$ is the solution of the Lyapunov equation (\ref{2.6}). From this, when
			\begin{equation}\label{3.2}
				2\Delta \alpha P_\alpha \prec Q + K^\top RK,
			\end{equation}
			$P_{\alpha} \in \mathbb{S}^{n}_{+}$ also satisfies
			$$
			(A_{\alpha^\prime}+BK)^\top P_\alpha + P_\alpha (A_{\alpha^\prime}+BK) + (C+DK)^\top P_\alpha (C+DK) \prec 0,
			$$
			then it follows from Lemma \ref{lem2.1} that $K \in \mathcal{K}^{(\alpha^\prime)}$. Thus, in the following, we aim to find sufficient conditions under which the partial order (\ref{3.2}) holds.
			
			The following result is directly taken from Theorem 4.2.2 in (\citet{horn2012matrix}, Page 176). For symmetric positive definite matrices $P_\alpha$, $Q$ and $R$, it holds that
			\begin{equation}\label{3.3}
				P_\alpha \preceq \lambda_{n}(P_\alpha)I_n \quad \text{and} \quad \lambda_{1}(Q)I_n \preceq Q \preceq Q+K^\top RK.
			\end{equation} 
			In addition, from (\ref{2.5}), one has
			\begin{equation*}
				J_{\alpha}(K) = \Tr(P_\alpha\Sigma_0) \geq \lambda_{1}(\Sigma_0)\Tr(P_\alpha) > \lambda_{1}(\Sigma_0)\lambda_{n}(P_\alpha).
			\end{equation*}
			Here the first inequality is due to Lemma \ref{lemA.1}, the last inequality follows from the positive definiteness of matrix $P_\alpha$. Therefore, 
			\begin{equation}\label{3.4}
				\lambda_n(P_\alpha) < \frac{J_{\alpha}(K)}{\lambda_1(\Sigma_0)}.
			\end{equation}
			Combining (\ref{3.3}) and (\ref{3.4}), the partial order (\ref{3.2}) holds when 
			\begin{equation}\label{3.5}
				2 \Delta \alpha \frac{J_{\alpha}(K)}{\lambda_{1}(\Sigma_0)} \leq \lambda_{1}(Q).
			\end{equation}
			Solving inequality (\ref{3.5}), one gets
			\begin{equation}\label{3.6}
				\Delta \alpha \leq \frac{\lambda_{1}(\Sigma_0)\lambda_{1}(Q)}{2J_{\alpha}(K)}.
			\end{equation}
			The upper bound in (\ref{3.1}) ensures that (\ref{3.6}) holds. Thus, we can obtain that $K \in \mathcal{K}^{(\alpha^\prime)}$. Next we shall show that $J_{\alpha^\prime}(K) \leq \zeta J_\alpha (K)$.
			
			Let $Y_{\alpha}$ be the solution to 
			\begin{equation*}
				(A_{\alpha}+BK)Y_{\alpha} + Y_{\alpha}(A_{\alpha}+BK)^\top + (C+DK)Y_{\alpha}(C+DK)^\top + \Sigma_0 = 0.
			\end{equation*}
			It follows from Lemma \ref{lemA.2} and (\ref{2.5}) that
			\begin{equation*}
				J_{\alpha}(K) = \Tr(P_{\alpha}\Sigma_0) = \Tr[Y_{\alpha}(Q + K^\top R K)].
			\end{equation*}
			Further, by Lemma \ref{lemA.1} and Lemma \ref{lemA.3}, it hold that
			\begin{equation*}
				J_{\alpha}(K) \geq \lambda_{1}(Q + K^\top R K)\Tr(Y_{\alpha})\geq \lambda_{1}(Q )\Tr(Y_{\alpha}).
			\end{equation*}
			Hence, 
			\begin{equation}\label{3.9}
				\Tr(Y_{\alpha}) \leq \frac{J_{\alpha}(K)}{\lambda_{1}(Q)}
			\end{equation}
			
			Under the condition (\ref{3.1}), $K \in \mathcal{K}^{(\alpha^\prime)}$. Then we consider the following Lyapunov equation:			
			\begin{equation}\label{3.7}
				\begin{aligned}
					(A_{\alpha^\prime}&+BK)^\top P_{\alpha^\prime} + P_{\alpha^\prime}(A_{\alpha^\prime}+BK)\\
					& + (C+DK)^\top P_{\alpha^\prime}(C+DK) + Q + {K}^\top R K = 0.
				\end{aligned}
			\end{equation}
			Subtracting (\ref{2.6}) from (\ref{3.7}) yields
			\begin{equation*}
				\begin{aligned}
					& (A_{\alpha}+BK)^\top (P_{\alpha^\prime}-P_{\alpha}) + (P_{\alpha^\prime}-P_{\alpha})(A_{\alpha}+BK) \\
					& \quad + (C+DK)^\top (P_{\alpha^\prime}-P_{\alpha})(C+DK) + 2\Delta \alpha P_{\alpha^\prime} = 0.
				\end{aligned}
			\end{equation*}
			Combining with (\ref{2.5}), the cost difference satisfies 
			\begin{equation}\label{3.8}
				\begin{aligned}
					J_{\alpha^\prime}(K) - J_{\alpha}(K) & =\Tr\left[ (P_{\alpha^\prime}-P_{\alpha})\Sigma_0 \right]\\
					& =\Tr\left(2\Delta \alpha P_{\alpha^\prime} Y_{\alpha} \right) \\
					&\leq 2\Delta \alpha \lambda_{n}(P_{\alpha^\prime})\Tr(Y_{\alpha}),
				\end{aligned}
			\end{equation} 
			here the first equality follows from Lemma \ref{lemA.2}, and the last inequality is due to Lemma \ref{lemA.1}.
			
			Inserting (\ref{3.1}), (\ref{3.4}) and (\ref{3.9}) into (\ref{3.8}), one has
			$$
			J_{\alpha^\prime}(K) - J_\alpha(K) \leq \frac{\zeta -1}{\zeta} J_{\alpha^\prime}(K).
			$$
			Rearranging terms yields
			$$
			J_{\alpha^\prime}(K) \leq \zeta J_\alpha(K)
			$$
			which completes the proof.
		\end{proof}
		
		From Lemma \ref{pro3.1}, we observe that the cost $J_\alpha$ increases by a factor of $\zeta$ when the discount factor $\alpha$ is decreased by $\Delta \alpha$. Whereas, as indicated by (\ref{3.1}), the decrement $\Delta \alpha$ decreases monotonically with increasing cost $J_\alpha$. Consequently, the decrement $\Delta \alpha$ may gradually vanish with the iteration of parameter $\alpha$. To maintain a sufficient magnitude of $\Delta \alpha$, we opt to reduce the objective value $J_\alpha$ to its current optimum $J_\alpha^{*}$ by using the PI method at each $\alpha$-reduction step. 
		
		To sum it up, at the $j$-th iteration, the discount method performs the following two procedures:
		\begin{itemize} \label{item}
			\item[\phantomsection\label{item:step}a.]  Solve $K_{j+1}$ via PI method such that 
			$$
			K_{j+1} = \operatorname{arg}\min \limits_{K} J_{\alpha_j}(K); 
			$$ 
			\item[\phantomsection\label{item:step}b.]  Update $\alpha_{j+1} = \alpha_j - \Delta \alpha_j$, where
			$$
			\Delta \alpha_j =  \frac{\lambda_{1}(\Sigma_0)\lambda_{1}(Q )}{2J_{\alpha_j}(K_{j+1})} \frac{\zeta-1}{\zeta}.
			$$
		\end{itemize}
		
		The detailed implementation of the discount method with the knowledge of the dynamics matrices $(A,B,C,D)$ is provided in Algorithm \ref{alg:Model-based discount}.
		
		\begin{algorithm}[H]
			\caption{Stabilizing known linear time-variant systems via discount method} \label{alg:Model-based discount}
			
			\textbf{Input:} Initial discount factor $\alpha_0$, initial feedback gain $\mathbf{O}_{m \times n}$\\ 
			\textbf{Initialization}: Set $\alpha \leftarrow \alpha_0$ and  $K \leftarrow K_0$\\
			\While{$\alpha > 0$}{
				Set \( i = 0 \) and $K^{(0)} \leftarrow K$
				
				\Repeat{\( \|P^{(i+1)}_\alpha - P^{(i)}_\alpha\| < \epsilon \)}{				
					Solve $P^{(i+1)}_\alpha$ from Lyapunov equation
					\begin{equation}\label{3.10}
						\begin{aligned}
							(A_{\alpha}&+BK^{(i)})^\top P^{(i+1)}_\alpha + P^{(i+1)}_\alpha (A_{\alpha}+BK^{(i)})  \\
							&+ (C+DK^{(i)})^\top P^{(i+1)}_\alpha (C+DK^{(i)}) + Q + {K^{(i)}}^\top R K^{(i)} =0.
						\end{aligned}
					\end{equation}
					
					Update $K^{(i+1)}$ via 
					\begin{equation}\label{3.11}
						K^{(i+1)} = -(R+D^\top P^{(i+1)}_\alpha D)^{-1}(B^\top P^{(i+1)}_\alpha+D^\top P^{(i+1)}_\alpha C).
					\end{equation}
					
					\( i \leftarrow i + 1 \)
				}
				Set $K \leftarrow K^{(i+1)}$; 
				
				Solve $P_\alpha$ from Lyapunov equation (\ref{2.6}); 
				
				Set $\alpha \leftarrow  \alpha - \Delta \alpha$, where
				\begin{equation*}
					\Delta \alpha =  \frac{\lambda_{1}(\Sigma_0)\lambda_{1}(Q )}{2\Tr(P_\alpha \Sigma_0)} \frac{\zeta-1}{\zeta}.
				\end{equation*}
			}
		\end{algorithm}
		
		Finally, we prove that the Algorithm \ref{alg:Model-based discount} synthesizes a stabilizer of the system $[A,C;B,D]$ after finitely many iterations. We first present the convergence of the PI algorithm (\ref{3.10}) $\&$ (\ref{3.11}). It theoretically ensures the feasibility of step \hyperref[item:step]{a.}. 
		
		\begin{pro}\label{thm3.1}
			Given a fixed discount factor $\alpha \geq 0$, suppose $K^{(0)}$ is a stabilizer for the system $[A_\alpha,C;B,D]$. Then
			\begin{itemize}
				\item[\phantomsection\label{item:custom}(a).] All the policies $\{K^{(i)}\}_{i=1}^{\infty}$ updated by (\ref{3.11}) are stabilizers.
				\item[\phantomsection\label{item:custom}(b).] There exists a unique solution $P^{(i+1)}_\alpha \in \mathbb{S}^{n}_{+}$ to (\ref{3.10}) at each step. 
				\item[\phantomsection\label{item:custom}(c).] The iteration $\{P^{(i+1)}_\alpha\}_{i=1}^{\infty}$ converges to the unique solution $P^{*}_\alpha \in \mathbb{S}^{n}_{+}$ of the following algebraic Riccati equation (ARE):
				\begin{equation}\label{3.12}
					\begin{aligned}
						A_\alpha^\top P^{*}_\alpha&+P^{*}_\alpha A_\alpha^\top+C^\top P^{*}_\alpha C +Q \\
						&-(P^{*}_\alpha B+C^\top P^{*}_\alpha D)(R+D^\top P^{*}_\alpha D)^{-1}(B^\top P^{*}_\alpha + D^\top P^{*}_\alpha C)=0.
					\end{aligned}
				\end{equation}
			\end{itemize}	
		\end{pro}
	
		\begin{proof}
			 The proof is reminiscent of Theorem 2.1-2.2 in (\citet{li2022stochastic}, Page 5013). The difference is that in \citep{li2022stochastic} the convergence of the PI algorithm was established for infinite-horizon SLQ problems with the deterministic initial state while here we study SLQ optimal control problems with the random initial state. Notably, the proof in \citep{li2022stochastic} depends entirely on the equivalent conditions in verifying the stabilizability (Lemma \ref{lem2.1}) which is independent of initial state. Therefore, the convergence result can be extended to this paper without requiring additional analysis.
		\end{proof}
		
		\begin{thm}\label{thm3.2}
			Let the initial discount factor $\alpha_0$ in Algorithm \ref{alg:Model-based discount} satisfy
			$$
				\alpha_0 > \frac{1}{2}\left(\lambda_{n}(A+A^\top) + \|C\|_2^2\:\right).
			$$
			Then Algorithm \ref{alg:Model-based discount} terminates after at most $\lceil \alpha_0/\tilde{\alpha} \rceil$ iterations and returns a stabilizing feedback gain for system $[A,C;B,D]$. Here, the constant $\tilde{\alpha}$ is 
			$$
				\tilde{\alpha} := \frac{\lambda_{1}(\Sigma_0)\lambda_{1}(Q )}{2J^{*}} \frac{\zeta-1}{\zeta},
			$$
			$\lceil \cdot \rceil$ denotes ceiling function that maps a real number to the smallest integer greater than or equal to this real number, and $J^{*}$ denotes the optimal value of the undiscounted SLQ problem.
		\end{thm}
		
		\begin{proof}
			The convergence of the PI method (\ref{3.10}) $\&$ (\ref{3.11}) established in Proposition \ref{thm3.1} guarantees that the decrement $\Delta \alpha_j$ in step \hyperref[item:step]{b.} satisfies
			$$
				\Delta \alpha_j = \frac{\lambda_{1}(\Sigma_0)\lambda_{1}(Q)}{2J_\alpha^{*}} \frac{\zeta-1}{\zeta}  \qquad \forall j \geq 1.
			$$
			Then Lemma \ref{lem2.3} implies that
			$$
				\Delta \alpha_j \geq \frac{\lambda_{1}(\Sigma_0)\lambda_{1}(Q )}{2J^{*}} \frac{\zeta-1}{\zeta} =: \tilde{\alpha}  \qquad \forall j \geq 1.
			$$
			Since $\Delta \alpha_j$ has a uniform lower bound $\tilde{\alpha}$, the discount method in Algorithm \ref{alg:Model-based discount} terminates after at most $\lceil \alpha_0/\tilde{\alpha} \rceil$ iterations.
			
			By Proposition \ref{thm3.1} \hyperref[item:custom]{(a).}, if $K^{(0)}$ is the stabilizer of the system $[A_{\alpha},C;B,D]$, the output of PI methods, $K^{(i+1)}$, is also the stabilizer of the system $[A_{\alpha},C;B,D]$. Further, because the decrement $\Delta \alpha$ in Algorithm \ref{alg:Model-based discount} satisfies (\ref{3.1}), it follows from Lemma \ref{pro3.1} that $K^{(i+1)}$ can stabilize the system $[A_{\alpha-\Delta \alpha},C;B,D]$. Significantly, this result holds true for each iteration. By Lemma \ref{lem3.1}, initial input $\mathbf{O}_{m \times n}$ is the stabilizer of the system $[A_{\alpha_0},C;B,D]$. Hence, the parameters in Algorithm \ref{alg:Model-based discount} ensure that the final output policy stabilizes the original stochastic LTI system.		
		\end{proof}
		
		\begin{remark}\label{remark2}
			Significantly, Proposition \ref{thm3.1} is derived solely through Lemma \ref{lem2.1}. And Lemma \ref{pro3.1} verifies that an originally stabilizing feedback gain $K \in \mathcal{K}^{(\alpha)}$ belongs to the contracted set $\mathcal{K}^{(\alpha - \Delta \alpha)}$ exclusively relying on Lemma \ref{lem2.1}. Further, Lemma \ref{lem2.1} implies that whether policy $K$ can stabilize stochastic LTI systems depends solely on system coefficient matrices, that is, the stabilizability of these systems is independent of their corresponding initial state. Thus stabilizers of the system (\ref{2.1}) derived from Algorithm \ref{alg:Model-based discount} can also stabilize other time-invariant stochastic linear dynamical control systems with same coefficient matrices and the deterministic initial state.
			
			Specially, we set the distribution of initial state as the standard normal distribution. Now, $\Sigma_0 = I_n$ and the decrement $\Delta \alpha$ in Algorithm \ref{alg:Model-based discount} is 
			$$
			\Delta \alpha =  \frac{\lambda_{1}(Q )}{2\Tr(P_\alpha)} \frac{\zeta-1}{\zeta}.
			$$
			Then Algorithm \ref{alg:Model-based discount} can synthesize stabilizers for stochastic LTI systems with the deterministic initial state, using only the coefficient matrices in state dynamics and cost functional.
		\end{remark}

		\noindent
		\textbf{B. Unknown model}
		
		From Algorithm \ref{alg:Model-based discount}, the model-free discount method can be established provided that solving the Lyapunov equation and updating policies can be implemented directly along the state and control trajectories. 
		
		To this end, we adopt adaptive dynamic programming (ADP) algorithm (\citet{werbos1974beyond}). This algorithm has been widely applied to solve optimal control problems in the model-free setting, such as continuous-time deterministic linear-quadratic control problems (\citet{jiang2012computational}), discrete-time SLQ optimal control problems (\citet{WANG2016379Infinite}), continuous-time SLQ problems with the deterministic initial state (\citet{zhang2023adaptive,zhang2024data}), among others.

		Inspired by these works, particularly (\citet{zhang2023adaptive,zhang2024data}), we now describe how to execute the policy evaluation step, (\ref{3.10}), and the policy improvement step, (\ref{3.11}), without requiring the knowledge of system matrices, thereby developing the model-free discount method. 		
		
		For completeness, we first restate Lemma 2 in (\citet{zhang2023adaptive}) due to its foundational role in constructing the ADP-based model-free PI algorithm. Based on this lemma, the system matrices $(A,B,C,D)$ required in (\ref{3.10}) and (\ref{3.11}) can be replaced by the observed state and input information.
		
		\begin{lem}\label{lem3.2}
			For any $i \geq 0$ and $\alpha \geq 0$, the solution $P^{(i+1)}_\alpha$ of (\ref{3.10}) and the policy $K^{(i+1)}$ updated by (\ref{3.11}) satisfies 
			\begin{equation}\label{3.13}
				\begin{aligned}
					&\mathbb{E} \big[\widetilde{X}(t+\Delta t)^\top P^{(i+1)}_\alpha \widetilde{X}(t+\Delta t)\big] - \mathbb{E}[\widetilde{X}(t)^\top P^{(i+1)}_\alpha \widetilde{X}(t)] \\
					& + 2\mathbb{E} \int_{t}^{t+\Delta t} \big( \widetilde{u}(s)- {K^{(i)}}\widetilde{X}(s) \big)^\top  M^{(i+1)}_\alpha \widetilde{X}(s) ds \\
					&- \mathbb{E} \int_{t}^{t+\Delta t}  \big( \widetilde{u}(s)- {K^{(i)}}\widetilde{X}(s) \big)^\top H^{(i+1)}_\alpha \big( \widetilde{u}(s)+ {K^{(i)}}\widetilde{X}(s) \big) ds  \\
					&\quad = -\mathbb{E} \int_{t}^{t+\Delta t}\widetilde{X}(s)^\top \big(Q + {K^{(i)}}^\top RK^{(i)} \big) \widetilde{X}(s) ds,
				\end{aligned}	
			\end{equation}
			where $0 \leq t < t+\Delta t < \infty$, $M^{(i+1)}_\alpha:= (R+D^\top P^{(i+1)}_\alpha D)K^{(i+1)} $; $H^{(i+1)}_\alpha := D^\top P^{(i+1)}_\alpha D$, and $\widetilde{X}(\cdot)$ is the solution of system (\ref{2.3}) with arbitrary admissible control $\widetilde{u}(\cdot)$.
		\end{lem}
		
		\begin{remark}
			Obviously, the equality (\ref{3.13}) holds for any admissible control $\widetilde{u}(\cdot)$ and its corresponding state $\widetilde{X}(\cdot)$. As a result, the ADP-based model-free PI algorithm that is building upon Lemma \ref{lem3.2} is an off-policy algorithm.
			
			Consistent with \citet{jiang2012computational}, we employ $\widetilde{u}(\cdot) = K^{(0)}\widetilde{X}(\cdot) + e(\cdot)$, where the exploration noise $e(\cdot)$ is the sum of sinusoidal signals with different frequencies. Note that this control law is limited to the implementation of the ADP-based model-free PI algorithm, we just consider the linear state-feedback control elsewhere in this paper.
		\end{remark}	
		  
		With those notations stated in Section \ref{sec1}, for any $V \in \mathbb{S}^{n}$, we define an operator $\Vech{(V)} \in \mathbb{R}^{\frac{1}{2}n(n+1)}$  as
		$$
		\Vech{(V)} :=  \big[v_{11},2v_{12},\cdots,2v_{1n},v_{22},2v_{23},\cdots,2v_{2n},\cdots,2v_{n-1,n},v_{nn}\big]^\top.
		$$ 
		From \citet{murray2002adaptive}, there exists a matrix $\Gamma \in \mathbb{R}^{n^2\times{\frac{1}{2}n(n+1)}}$ mapping $\Vech(V)$ to $\Vect(V)$, i.e. $\Gamma \Vech(V)=\Vect(V)$. For any $\nu \in \mathbb{R}^n$, one has
		$$
		\nu^\top V \nu= (\nu^\top \otimes \nu^\top)\Vect(V) = (\nu^\top \otimes \nu^\top)\Gamma\Vech(V) =: \mathcal{M}(\nu)^\top\Vech(V).
		$$
		Then by applying vectorization methods and Kronecker product theory, the term in (\ref{3.13}) can be rewritten as
		\begin{equation}\label{3.13a}
			\begin{aligned}
				&\mathbb{E}\left[\mathcal{M}(\widetilde{X}(t+\Delta t)) - \mathcal{M}(\widetilde{X}(t))\right]^\top \Vech(P^{(i+1)}_\alpha)  \\
				&+ 2\mathbb{E} \int_{t}^{t+\Delta t} (\widetilde{X}(s)^\top \otimes \widetilde{u}(s)^\top ) \\
				&\quad - (\widetilde{X}(s)^\top \otimes \widetilde{X}(s)^\top ) (I_n \otimes {K^{(i)}}^\top ) ds \Vect(M^{(i+1)}_\alpha) \\
				& - \mathbb{E} \int_{t}^{t+\Delta t} \mathcal{M}(\widetilde{u}(s))^\top - \mathcal{M}(K^{(i)}\widetilde{X}(s))^\top ds \Vech(H^{(i+1)}_\alpha) \\
				& \quad = -\mathbb{E} \int_{t}^{t+\Delta t}\widetilde{X}(s)^\top \big(Q + {K^{(i)}}^\top RK^{(i)} \big) \widetilde{X}(s) ds.
			\end{aligned}
		\end{equation}
		
		Further, we define
		\begin{equation*}
			\begin{aligned}
				&\mathbf{\Xi} = \mathbb{E} \left[\mathcal{M}(\widetilde{X}_1(t_0)) - \mathcal{M}(\widetilde{X}_1(0)), \mathcal{M}(\widetilde{X}_2(t_0)) - \mathcal{M}(\widetilde{X}_2(0)),  \right. \\
				& \qquad \left. \cdots, \mathcal{M}(\widetilde{X}_{l-1}(t_0)) - \mathcal{M}(\widetilde{X}_{l-1}(0)), \mathcal{M}(\widetilde{X}_l(t_0)) - \mathcal{M}(\widetilde{X}_l(0)) \right] ^\top,
			\end{aligned}
		\end{equation*}
		\begin{equation*}
			\mathbb{I}_{\mathbf{xu}} = \mathbb{E} \left[ \int_{0}^{t_0} \widetilde{X}_1(s) \otimes \widetilde{u}_1(s) ds, \int_{0}^{t_0} \widetilde{X}_2(s) \otimes \widetilde{u}_2(s) ds, \cdots , \int_{0}^{t_0}\widetilde{X}_l(s) \otimes \widetilde{u}_l(s) ds \right]^\top,
		\end{equation*}
		\begin{equation*}
			\mathbb{I}_{\mathbf{xx}} = \mathbb{E} \left[ \int_{0}^{t_0} \widetilde{X}_1(s) \otimes \widetilde{X}_1(s) ds, \int_{0}^{t_0} \widetilde{X}_2(s) \otimes \widetilde{X}_2(s) ds, \cdots , \int_{0}^{t_0}\widetilde{X}_l(s) \otimes \widetilde{X}_l(s) ds \right]^\top,
		\end{equation*}
		\begin{equation*}
			\mathbb{M}_{\mathbf{u}} = \mathbb{E} \left[ \int_{0}^{t_0} \mathcal{M}(\widetilde{u}_1(s)) ds, \int_{0}^{t_0} \mathcal{M}(\widetilde{u}_2(s)) ds, \cdots , \int_{0}^{t_0} \mathcal{M}(\widetilde{u}_l(s)) ds \right]^\top,
		\end{equation*}
		\begin{equation*}
			\mathbb{M}^{(i)}_{\mathbf{kx}} = \mathbb{E} \left[ \int_{0}^{t_0} \mathcal{M}(K^{(i)}\widetilde{X}_1(s)) ds, \int_{0}^{t_0} \mathcal{M}(K^{(i)}\widetilde{X}_2(s)) ds, \cdots , \int_{0}^{t_0} \mathcal{M}(K^{(i)}\widetilde{X}_l(s))ds \right]^\top,
		\end{equation*}
		\begin{equation*}
			\begin{aligned}
				&\mathbb{J}^{(i)}_{\mathbf{k}} = -\mathbb{E} \left[\int_{0}^{t_0} \widetilde{X}_1(s)^\top (Q + {K^{(i)}}^\top R K^{(i)}) \widetilde{X}_1(s) ds, \int_{0}^{t_0} \widetilde{X}_2(s)^\top (Q + {K^{(i)}}^\top R K^{(i)}) \widetilde{X}_2(s) ds,  \right. \\
				& \qquad \qquad \left. \cdots, \int_{0}^{t_0} \widetilde{X}_l(s)^\top (Q + {K^{(i)}}^\top R K^{(i)}) \widetilde{X}_l(s) ds \right] ^\top,
			\end{aligned}
		\end{equation*}
		where $t_0 > 0$ is an arbitrary time point, $X_h(\cdot) \equiv X_h(\cdot\:; 0, x_h, u_h(\cdot)) \: (1\leq h \leq l)$ denotes the state trajectory with different initial state $x_h \in \mathcal{H}$. 
		
		Then, for any given stabilizing policy $K^{(i)}$ ($i \geq 0$) , (\ref{3.13a}) implies that 	
		\begin{equation*}
			\mathbf{\Phi}_i
			\begin{bmatrix}
				\Vech(P^{(i+1)}_\alpha) \\
				\Vect(M^{(i+1)}_\alpha)\\
				\Vech(H^{(i+1)}_\alpha)
			\end{bmatrix} = \mathbb{J}^{(i)}_{\mathbf{k}}
		\end{equation*}
		where
		\begin{equation*}
			\mathbf{\Phi}_i = \left[\mathbf{\Xi}, 2\big( \mathbb{I}_{\mathbf{xu}}-\mathbb{I}_{\mathbf{xx}}(I_n \otimes {K^{(i)}}^\top) \big), \mathbb{M}^{(i)}_{\mathbf{kx}} - \mathbb{M}_{\mathbf{u}} \right].
		\end{equation*}
		Under the rank condition (specified in Lemma 3, \citet{zhang2023adaptive}) that ensures $\mathbf{\Phi}_i$ ($i \geq 0$) has full column rank, we can obtain unique $P^{(i+1)}_\alpha$, $M^{(i+1)}_\alpha$ and $H^{(i+1)}_\alpha$ by directly calculating
		\begin{equation}\label{3.15}
			\begin{bmatrix}
				\Vech(P^{(i+1)}_\alpha) \\
				\Vect(M^{(i+1)}_\alpha)\\
				\Vech(H^{(i+1)}_\alpha)
			\end{bmatrix} =({\mathbf{\Phi}_i}^\top\mathbf{\Phi}_i)^{-1}{\mathbf{\Phi}_i}^\top \mathbb{J}^{(i)}_{\mathbf{k}},
		\end{equation}
		and then $K^{(i+1)}$ is updated by
		\begin{equation}\label{3.15a}
			K^{(i+1)} = (R+H^{(i+1)}_\alpha)^{-1}M^{(i+1)}_\alpha.
		\end{equation}

		At this point, given a stabilizer $K^{(i)}$, we can execute one-step policy iteration defined by (\ref{3.10}) $\&$ (\ref{3.11}) in the model-free setting. Naturally, step \hyperref[item:step]{a.} in the discount method can be implemented directly along the sampled state and control trajectories. Thus, the remaining piece that is required to develop the model-free discount method is the way to determine the decrement $\Delta \alpha_j$	in step \hyperref[item:step]{b.} without requiring the knowledge of the system matrices. 
		
		For the calculation of decrement $\Delta \alpha_j$, multiple model-free approaches are available for computing the value of the cost function $J_{\alpha_j}(\cdot)$ at the point $K_{j+1}$. One way is to directly evaluate $J_{\alpha_j}(K_{j+1})$ via (\ref{2.4}). This approach is performed through the state trajectory samples generated from simulating system (\ref{2.3}) under the input $\widetilde{u}(\cdot) = K_{j+1}\widetilde{X}(\cdot)$. 
		
		From (\ref{2.5}), we observe that the computation of the value $J_{\alpha_j}(K_{j+1})$ can be simplified to calculate corresponding $P_{\alpha_j}$. Then, another way is to solve $P_{\alpha_j}$ from the identity
		\begin{equation*}
			\begin{aligned}
				&\mathbb{E}[{\widetilde{X}(t)}^\top P_{\alpha_j}\widetilde{X}(t)] - \mathbb{E}[{\widetilde{X}(t+\Delta t)}^\top P_{\alpha_j}\widetilde{X}(t+\Delta t)] \\
				& = \mathbb{E} \int_{t}^{t+\Delta t} {\widetilde{X}(s)}^\top \left( Q + K_{j+1}^\top R K_{j+1} \right) \widetilde{X}(s) ds
			\end{aligned}
		\end{equation*}
		where $\widetilde{X}(\cdot)$ is the solution of (\ref{2.3}) with $\widetilde{u}(\cdot) = K_{j+1}\widetilde{X}(\cdot)$. This way follows from the work of \citet{li2022stochastic}. The third way employs the aforementioned ADP method to determine $P_{\alpha_j}$. Specifically, calculate the matrix $\mathbb{J}^{(i)}_{\mathbf{k}}$ and  $\mathbb{M}^{(i)}_{\mathbf{kx}}$ corresponding to $K_{j+1}$, and then obtain $P_{\alpha_j}$ from (\ref{3.15}).
		
		Notably, the first and the second approach require new state trajectories collected through running system (\ref{2.3}) under the control input $\widetilde{u}(\cdot) = K_{j+1}\widetilde{X}(\cdot)$, whereas the third approach can reuse directly existing state and input data collected during the execution of the ADP-based PI method. Hence, the third approach is adapted in this paper, despite inevitably yielding unnecessary by-product $M_\alpha$ and $H_\alpha$. Now, we present the model-free discount method in Algorithm \ref{alg:Model-free discount}.
		
		\begin{algorithm}[H]
			\caption{The model-free discount method} \label{alg:Model-free discount}
			
			\textbf{Input:} Initial discount factor $\alpha_0$, initial feedback gain $\mathbf{O}_{m \times n}$\\ 
			\textbf{Initialization}: Set $\alpha \leftarrow \alpha_0$ and  $K \leftarrow K_0$\\
			\While{$\alpha > 0$}{
				Set \( i = 0 \) and $K^{(0)} \leftarrow K$
				
				\textbf{Data Collection}: Collect state data $\widetilde{X}(\cdot)$ and control data $\widetilde{u}(\cdot)$ by running system (\ref{2.3}) with 
				$
				\widetilde{u}(\cdot) = K^{(0)}\widetilde{X}(\cdot) + e(\cdot)
				$ 
				on time interval $[t_0,t_l]$, where $e(\cdot)$ is the exploration noise. 
				
				Compute $\mathbf{\Xi}$, $\mathbb{I}_{\mathbf{xx}}$, $\mathbb{I}_{\mathbf{xu}}$, $\mathbb{M}_{\mathbf{u}}$. 
				
				\Repeat{\( \|P^{(i+1)}_\alpha - P^{(i)}_\alpha\| < \epsilon \)}{
					Compute $\mathbb{J}^{(i)}_{\mathbf{k}}$, $\mathbb{M}^{(i)}_{\mathbf{kx}}$.
					
					Solve $P^{(i+1)}_\alpha$, $M^{(i+1)}_\alpha$, $H^{(i+1)}_\alpha$ from (\ref{3.15}).
					
					Update $K^{(i+1)}$ via 
					$
					K^{(i+1)} = (R+H^{(i+1)}_\alpha)^{-1}M^{(i+1)}_\alpha.
					$
					
					\( i \leftarrow i + 1 \)
				}
				Compute $\mathbb{J}^{(i)}_{\mathbf{k}}$, $\mathbb{M}^{(i)}_{\mathbf{kx}}$ and solve $P^{(i+1)}_\alpha$  from (\ref{3.15})
				
				Set $K \leftarrow K^{(i+1)}$ and $P_\alpha \leftarrow P^{(i+1)}_\alpha$
				
				Update $\alpha \leftarrow \alpha - \frac{\lambda_{1}(\Sigma_0)\lambda_{1}(Q )}{2\Tr(P_\alpha \Sigma_0)} \frac{\zeta-1}{\zeta}$.
			}
		\end{algorithm}
		
		Finally, we discuss the feasibility of Algorithm \ref{alg:Model-free discount}. 
		\begin{thm}
			Under the conditions of Theorem \ref{thm3.2}, Algorithm \ref{alg:Model-free discount} returns a stabilizing feedback gain for system (\ref{2.1}), using the same number of iterations as Algorithm \ref{alg:Model-based discount}.
		\end{thm}
		\begin{proof}
			By (Theorem 4, \citet{zhang2023adaptive}), performing policy evaluation (\ref{3.10}) and policy improvement (\ref{3.11}) in Algorithm \ref{alg:Model-based discount} is equivalent to obtaining $P^{(i+1)}_\alpha$, $M^{(i+1)}_\alpha$ and $H^{(i+1)}_\alpha$ from (\ref{3.15}) and updating $K^{(i+1)}$ via (\ref{3.15a}) in Algorithm \ref{alg:Model-free discount}. Hence, the conclusion established in Theorem \ref{thm3.2} for Algorithm \ref{alg:Model-based discount} remains valid for Algorithm \ref{alg:Model-free discount}. 
		\end{proof}
		
		\begin{remark}
			As stated in Remark \ref{remark2}, after setting $\Sigma_0 = I_n$, the Algorithm \ref{alg:Model-free discount}, originally designed for system (\ref{2.1}) with the random initial state, can also stabilize stochastic LTI systems with same system matrices and the deterministic initial state. Notably, at this case, the collected state and control trajectories correspond to the deterministic initial state. Moreover, the theoretical results established in \citet{zhang2023adaptive} show the feasibility of policy evaluation and policy improvement under the deterministic initial state.
		\end{remark}
	
	\section{Numerical experiment}\label{sec4}
		By the Kronecker product theory, the Lyapunov equation (\ref{3.10}) implies that
		\begin{equation}\label{4.1}
			\begin{aligned}
				&\left[ I_n \otimes (A_\alpha + BK^{(i)})^\top + (A_\alpha + BK^{(i)})^\top \otimes I_n  \right. \\
				& \qquad \left. + (C + DK^{(i)})^\top \otimes (C + DK^{(i)})^\top  \right] \Vect(P_{\alpha}^{(i+1)}) \\
				& \qquad  = - \Vect(Q + {K^{(i)}}^\top RK^{(i)}).
			\end{aligned}
		\end{equation}
		Because $K^{(i)}$ is a stabilizer of the system $[A_\alpha, C;B,D]$, Lemma \ref{lem2.1} ensures the existence and uniqueness of the solution $\Vect(P_{\alpha})$ to equation (\ref{4.1}). Then, in the implementation of Algorithm \ref{alg:Model-based discount}, we solve equation (\ref{4.1}) for $\Vect(P_{\alpha}^{(i+1)})$, thereby obtaining the solution $P_{\alpha}^{(i+1)}$ of the Lyapunov equation (\ref{3.10}).
		
		In the implementation of Algorithm \ref{alg:Model-free discount}, after we obtain $\mathbb{N}$ state/control data with the data sampled at $\mathbb{Q}$ equally spaced time points over the interval $[0,t_0]$ ($0 = s_0 < \cdots < s_q < \cdots< s_\mathbb{Q} = t_0$), where $\mathbb{N}$ and $\mathbb{Q}$ are large enough, similar to \cite{li2022stochastic}, we approximate $\mathbb{E}[\mathcal{M}(\widetilde{X}_h(t_0))]$ as 
		$$
		\mathbb{E}[\mathcal{M}(\widetilde{X}_h(t_0))] \approx \frac{1}{\mathbb{N}} \sum_{k = 1}^{\mathbb{N}} \mathcal{M}(\widetilde{X}_h^{(k)}(t_0))
		$$
		and calculate $\mathbb{E} \int_{0}^{t_0} \widetilde{X}_h(s) \otimes \widetilde{u}_h(s) ds$ in $\mathbb{I}_{\mathbf{xu}}$ as 
		$$
		\mathbb{E} \int_{0}^{t_0} \widetilde{X}_h(s) \otimes \widetilde{u}_h(s) ds \approx \frac{1}{\mathbb{N}} \sum_{k = 1}^{\mathbb{N}} \left[ \sum_{q=0}^{\mathbb{Q}-1} \left(\widetilde{X}^{(k)}_h(s_q) \otimes \widetilde{u}_h(s_q)\right) \cdot \frac{t_0}{\mathbb{Q}} \right].
		$$
		Similarly, we can approximate $\mathbb{I}_{\mathbf{xx}}$, $\mathbb{M}_{\mathbf{u}}$, $\mathbb{M}^{(i)}_{\mathbf{kx}}$, $\mathbb{J}^{(i)}_{\mathbf{k}}$. In addition, matrix $\Sigma_0$ can be directly approximated using
		$$
			\Sigma_0 \approx \frac{1}{\mathcal{N}} \sum_{r = 1}^{\mathcal{N}} \widetilde{X}^{(r)}(0)\widetilde{X}^{(r)}(0)^\top.
		$$
		where $\widetilde{X}^{(r)}(0)$ is randomly sampled from the distribution of the initial state, and $\mathcal{N}$ is large enough.
		
		Following this, we perform the Algorithm \ref{alg:Model-based discount} and Algorithm \ref{alg:Model-free discount} on a linear system with two-dimensional state space and one-dimensional control input for illustration. 
		
		We set
		\begin{equation*}
			A = 
			\begin{bmatrix} 
				3 & 6 \\ 
				11 & -7
			\end{bmatrix},
			B = 
			\begin{bmatrix} 
				7 \\ 
				2
			\end{bmatrix},
			C = 
			\begin{bmatrix} 
				0.6 & 0.1 \\ 
				-0.3 & 0.7
			\end{bmatrix},
			D = 
			\begin{bmatrix} 
				0.2 \\
				0.1
			\end{bmatrix}.
		\end{equation*}
		and choose the initial state distribution as the standard normal distribution, which implies $\Sigma_0$ is the identity matrix. In addition, we choose $Q = \diag(7, 3)$ and $R = 2$ in the LQR cost.
		
		By Lemma \ref{lem3.1}, we set the initial discount factor $\alpha_0 = 9$. We then independently sample multiple initial states from the standard normal distribution, and simulate the stochastic linear system (\ref{2.4a}) via the Euler-Maruyama scheme under the feedback gain $\mathbf{O}_{m\times n}$. Figure \ref{fig:state_trajectory} shows the mean value of the resulting state trajectories. Seen from Figure \ref{fig:state_trajectory}, the state trajectory tends to a neighborhood of zero as time goes to infinity, confirming $\mathbf{O}_{m\times n}$ as the stabilizer of the system $[A_{\alpha_0},C;B,D]$.
				
		\begin{figure}[t]
			\centering
			\includegraphics[width=0.7\textwidth]{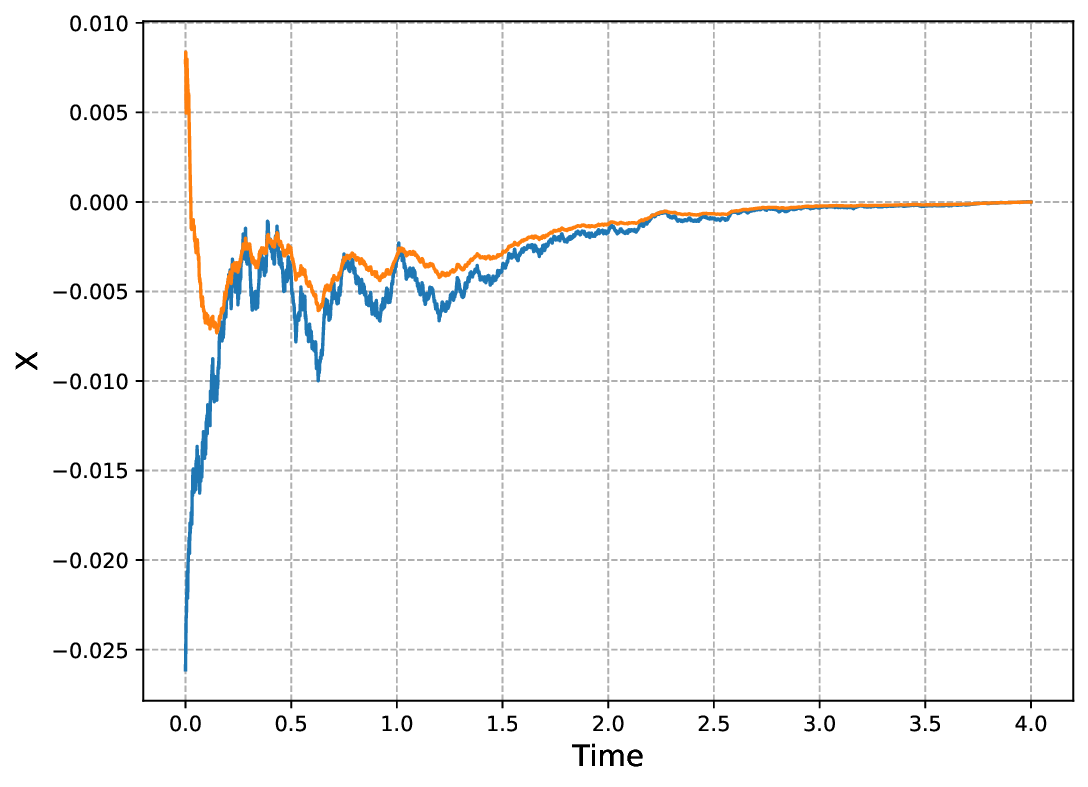}
			\caption{The average performance of state trajectory}
			\label{fig:state_trajectory}
		\end{figure}
		
		Implementing Algorithm \ref{alg:Model-based discount} from the initial values $\alpha_0 = 9$ and $K_0 = \mathbf{O}_{m \times n}$, we observe that this algorithm returns a stabilizing feedback gain $K = (-2.731,-1.027)$ with $\zeta = 10$ in only 5 steps. The dependence of the discount factor $\alpha$ on the number of iterations is shown in Figure \ref{fig:alpha_algorithm1}.
		
		\begin{figure}[t]
			\centering
			\includegraphics[width=0.7\textwidth]{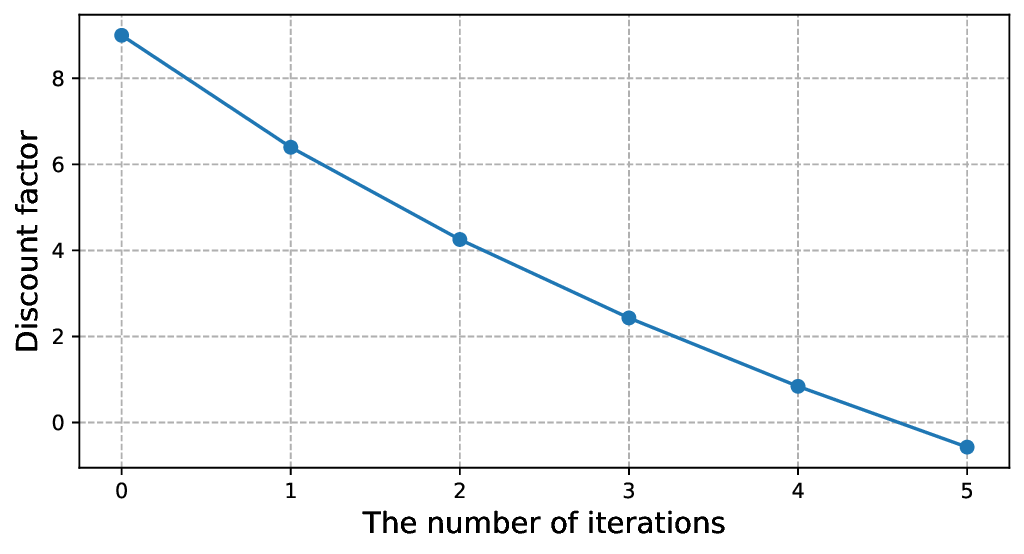}
			\caption{The iteration of discount factor $\alpha$ in Algorithm \ref{alg:Model-based discount}}
			\label{fig:alpha_algorithm1}
		\end{figure}
		
		In the model-free setting, we set the number of trajectories $\mathbb{N} = 10000$ and the number of grid $\mathcal{J} = 100$. State and input information are collected over each closed interval of 1-second length. In this case, if the knowledge of $\lambda_{n}(A+A^\top)$ and $\|C\|_2$ is available to set $\alpha_0 = 9$, then the performance of Algorithm \ref{alg:Model-free discount} is similar to that of Algorithm \ref{alg:Model-based discount}. Specifically, with parameter $\zeta = 10$, Algorithm \ref{alg:Model-free discount} terminates at the 5th step, and synthesizes a stabilizer $K = (-2.649,-0.849)$.
		
		However, the knowledge of $\lambda_{n}(A+A^\top)$ and $\|C\|_2$ may be unavailable in the model-free setting. In this paper, we choose a sufficiently large discount factor $\alpha = 200$ to satisfy the condition $\alpha_0 > \frac{1}{2}\left(\lambda_{n}(A+A^\top) + \|C\|_2^2\:\right)$ in Lemma \ref{lem3.1}. After 14 iterations, the discount factor $\alpha$ decreases to zero. The detailed iterative process of parameter $\alpha$ is shown in Figure \ref{fig:alpha_algorithm2}.
		
		\begin{figure}[t]
			\centering
			\includegraphics[width=0.7\textwidth]{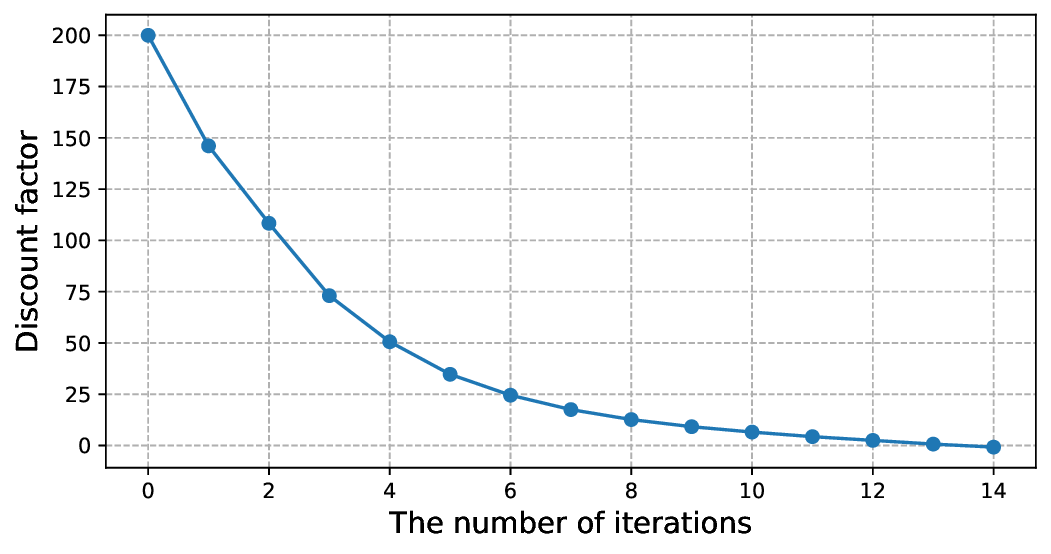}
			\caption{The iteration of discount factor $\alpha$ in Algorithm \ref{alg:Model-free discount}}
			\label{fig:alpha_algorithm2}
		\end{figure}
		
		In addition,  \citet{ozaslan2022computing} provide another way to determine parameter $\alpha_0$. They first set $\alpha_0 = 0$, then gradually increase it until the cost estimate achieves a certain threshold. Anyway, those methods are significantly easier to implement than selecting an initial stabilizer $K_0$ via numerical experiment.
		
	\section{Conclusion}
		In this paper, we first develop a discount method to compute stabilizing feedback gains for stochastic LTI systems with known system matrices. We subsequently extend this method to the case when the system matrices are completely unknown, using the idea of ADP algorithm. Both rigorous proof and numerical simulation guarantee the effectiveness of algorithms.
	
		Inspired by the work of \citet{perdomo2021stabilizing}, the discount method  developed in this paper might be extended to more complex stochastic systems. In addition, sample complexity of this algorithm is worth further consideration.

	\appendix
	\renewcommand{\appendixname}{}
	
	\section{Some helpful lemmas}
	
		Lemma \ref{lemA.1} and Lemma \ref{lemA.2} correspond to Lemma A.2 and Lemma A.3 in (\citet{Zhang2025Convergence}, Page 18), respectively. Given that these lemmas are repeatedly employed in this paper, we restate them here.
		
		\begin{lem}\label{lemA.1}
			For arbitrary positive semidefinite $M_1, M_2 \in \overline{\mathbb{S}_{+}^{n}}$, it holds that
			$$
			\lambda_1(M_1)\Tr(M_2) \leq \Tr(M_1M_2) \leq \lambda_n(M_1)\Tr(M_2).
			$$
		\end{lem}
		\begin{lem}\label{lemA.2}
			Suppose $K$ is a mean-square stabilizer of the system $[A,C;B,D]$. Let $P$ and $Y$ be the solution of the dual Lyapunov equations 
			\begin{equation*} 
				(A + BK)^\top P + P(A + BK) + (C + DK)^\top P(C + DK) + \Lambda = 0, 
			\end{equation*}
			\begin{equation*}
				(A + BK)Y + Y(A + BK)^\top + (C + DK)Y(C + DK)^\top + V = 0.
			\end{equation*}		 
			Then $\Tr(PV) =\Tr(Y\Lambda)$.
		\end{lem}
		
		\begin{lem}\label{lemA.3}
			For arbitrary symmetric matrices $M_1, M_2 \in \mathbb{S}^{n}$, it holds that
			$$
			\lambda_{n}(M_1+M_2) \leq \lambda_{n}(M_1) + \lambda_{n}(M_2)
			$$
			$$
			\lambda_{1}(M_1+M_2) \geq \lambda_{1}(M_1) + \lambda_{1}(M_2)
			$$
		\end{lem}
		\begin{proof}
			By (\citet{horn2012matrix}, Theorem 4.2.2, Page 176),
			\begin{equation*}
				\begin{aligned}
					\lambda_{n}(M) = \underset{x^\top x = 1}\Max x^\top M \:x \\
					\lambda_{1}(M) = \underset{x^\top x = 1}\Min x^\top M \:x
				\end{aligned}
			\end{equation*}
			hold for an arbitrary symmetric matrix $M \in \mathbb{S}^{n}$.

			For any $x \in \mathbb{R}^n$ with $x^\top x = 1$, it holds that
			$$
			x^\top (M_1 + M_2) x \leq \Max(x^\top M_1 x) + \Max(x^\top M_2 x),
			$$
			by the arbitrariness of unit vector $x$, one has 
			\begin{equation*}
				\underset{x^\top x = 1}\Max x^\top (M_1 + M_2) \:x \leq \underset{x^\top x = 1}\Max x^\top M_1 \:x + \underset{x^\top x = 1}\Max x^\top M_2 \:x,
			\end{equation*}
			then
			$$
			\lambda_{n}(M_1+M_2) \leq \lambda_{n}(M_1) + \lambda_{n}(M_2).
			$$
			A similar proof applies to the smallest eigenvalue, thereby establishing the full result.
		\end{proof}

\bibliographystyle{elsarticle-num-names} 
\bibliography{ref}

\end{document}